\documentclass[11pt]{article}

\usepackage{amsrefs}
\usepackage{amsmath,amssymb,amsthm,enumerate,tikz}
\usepackage[all,cmtip]{xy}
\usepackage[applemac]{inputenc}

\def \1{{\bf 1}}

\def \al{\alpha}

\def \bs{\backslash}
\def \C{{\mathbb C}}
\def \CA{{\cal A}}
\def \CB{{\cal B}}

\def \CE{{\cal E}}
\def \CF{{\cal F}}

\def \CH{\mathcal H}
\def \CL{{\cal L}}
\def \CN{{\cal N}}
\def \CO{{\cal O}}

\def \CR{{\cal R}}
\def \CS{{\cal S}}

\def \CU{\mathcal U}
\def \cont{\mathrm{cont}}

\def \der{\mathrm{der}}

\def \eps{\varepsilon}

\def \Ga{\Gamma}

\def \GL{\operatorname{GL}}
\def \ga{\gamma}
\def \g{\mathfrak g}
\def \Hom{\operatorname{Hom}}
\def \Id{{\rm Id}}
\def \Im{\operatorname{Im}}
\def \Ind{\operatorname{Ind}}

\def \la{\lambda}

\def \Lie{\operatorname{Lie}}

\def \mqed{\tag*\qedhere}
\def \N{{\mathbb N}}

\def \ol{\overline}
\def \op{\mathrm{op}}

\def \Om{\Omega}

\def \Q{{\mathbb Q}}

\def \R{{\mathbb R}}

\def \SL{\operatorname{SL}}

\def \T{{\mathbb T}}
\def \tr{\operatorname{tr}}
\def \triv{\mathrm{triv}}

\def \what{\widehat}

\def \({\left(}
\def \){\right)}

\newcommand{\e}
[1]{\emph{#1}\index{#1}}

\newcommand{\mat}
[4]{\(\begin{matrix}#1 & #2 \\ #3 & #4\end{matrix}\)}

\newcommand{\norm}
[1]{\left|\hspace{-1pt}\left|#1\right|\hspace{-1pt}\right|}
 
\newcommand{\smat}
[4]{\(\begin{smallmatrix}#1 & #2 \\ #3 & #4\end{smallmatrix}\)}

\renewcommand{\sp}
[1]{\left\langle #1\right\rangle}

\newcommand{\tto}
[1]{\stackrel{#1}{\longrightarrow}}

\newtheorem{theorem}{Theorem}[section]

\newtheorem{lemma}[theorem]{Lemma}

\newtheorem{proposition}[theorem]{Proposition}
\newtheorem{exmples}[theorem]{Examples}
\newenvironment{examples}[0]{\begin{exmples}{\ }\\ 	\vspace{-20pt}\nopagebreak[4]
	\begin{itemize}\rm}{\end{itemize}\end{exmples}\vspace{5pt}}

\theoremstyle{definition}
\newtheorem{definition}[theorem]{Definition}
\newtheorem{example}[theorem]{Example}
\newtheorem{remark}[theorem]{Remark}

\begin{document}

\pagestyle{myheadings} \markright{TRACE CLASS GROUPS}

\title{Trace class groups\\ \ \\ \small
Journal of Lie Theory 26, No. 1, 269-291 (2016)}
\author{Anton Deitmar \& Gerrit van Dijk}
\date{}
\maketitle

{\bf Abstract:}
A representation $\pi$ of a locally compact group $G$ is called \e{trace class}, if for every test function $f$ the induced operator $\pi(f)$ is a trace class operator. The group $G$ is called \e{trace class}, if every $\pi\in\what G$ is trace class.
In this paper we give a survey of what is known about trace class groups and ask for a simple criterion to decide whether a given group is trace class.
We show that trace class groups are type I and give a criterion for semi-direct products to be trace class and show that a representation $\pi$ is trace class if and only if $\pi\otimes\pi'$ can be realized in the space of distributions.

$$ $$

\tableofcontents

\newpage
\section*{Introduction}

Let $G$ be a locally compact group and $\pi$ an irreducible unitary representation.
Let $f\in C_c^\infty(G)$ be a test function and form the operator $\pi(f)=\int_Gf(x)\pi(x)\,dx$.
For applications of the trace formula, it is 
often necessary to evaluate the trace of $\pi(f)$.
For this one first has to make sure that all operators $\pi(f)$ are indeed trace class.
We call the group $G$ a \e{trace class group} if $\pi(f)$ is always trace class. 
It has been shown by Kirillov \cite{Kiri}, that nilpotent Lie groups are trace class and by Harish-Chandra \cite{Hari}, that every semi-simple Lie group with finitely many connected components and finite center is trace class.

In the present paper we give a survey of what is known about trace class groups. We show that trace class groups are type I, but that the converse does not hold.
We also show that the trace class property persists in certain exact sequences.
We present a proof that being trace class is equivalent to the fact that every irreducible unitary representation can be realized in the space of distributions.
In the case of discrete groups, we give a full classification of trace class groups.

We thank Karl-Hermann Neeb for his helpful comments including a different proof of Proposition \ref{lem1.3}.
We also thank the referee for a formidable job.
He has read the paper very carefully and has given us numerous hints leading to improvements of the paper.
For instance, the current proof of Theorem \ref{thm1.7} is due to the referee, as is the converse direction of the proof of Proposition \ref{prop1.9} (b).

\section{Trace class}

On any locally compact group $G$ there is an accepted notion of a space $C_c^\infty(G)$ of test functions given by Bruhat in \cite{Bruhat}. 
This space comes with a natural topology making it a locally convex space.
For the convenience of the reader we will briefly repeat the definition of the space and its topology.

\begin{definition}
First, if $L$ is a Lie group, then $C_c^\infty(L)$ is defined as the space of all infinitely differentiable functions of compact support on $L$.
The space $C_c^\infty(L)$ is the inductive limit of all $C_K^\infty(L)$, where $K\subset L$ runs through all compact subsets of $L$ and $C_K^\infty(L)$ is the space of all smooth functions supported inside $K$. The latter is a Fr\'echet space equipped with the supremum norms over all derivatives. Then $C_c^\infty(L)$ is equipped with the inductive limit topology in the category of locally convex spaces as defined in \cite{Schaef}, Chap II, Sec. 6.

Next, suppose the locally compact group $H$ has the property that $H/H^0$ is compact, where $H^0$ is the connected component.
Let $\CN$ be the family of all normal closed subgroups $N\subset H$ such that $H/N$ is a Lie group with finitely many connected components.
We call $H/N$ a \e{Lie quotient} of $H$.
Then, by \cite{MZ}, the set $\CN$ is directed by inverse inclusion and 
$$
H\cong \lim_{\substack{\leftarrow\\ N}}H/N,
$$
where the inverse limit runs over the set $\CN$.
So $H$ is a projective limit of Lie groups.
The space $C_c^\infty(H)$ is then defined to be the sum of all spaces $C_c^\infty(H/N)$ as $N$ varies in $\CN$.
Then $C_c^\infty(H)$ is the inductive limit over all $C_c^\infty(L)$ running over all Lie quotients $L$ of $H$ and so $C_c^\infty(H)$ again is equipped with the inductive limit topology in the category of locally convex spaces.

Finally to the general case.
By \cite{MZ} one knows that every locally compact group $G$ has an open subgroup $H$ such that $H/H^0$ is compact, so $H$ is a projective limit of connected Lie groups in a canonical way.
A Lie quotient of $H$ then is called a \e{local Lie quotient} of $G$.
We have the notion $C_c^\infty(H)$ and for any $g\in G$ we define $C_c^\infty(gH)$ to be the set of functions $f$ on the coset $gH$ such that $x\mapsto f(gx)$ lies in $C_c^\infty(H)$.
We then define $C_c^\infty(G)$ to be the sum of all $C_c^\infty(gH)$, where $g$ varies in $G$.
Then $C_c^\infty(G)$  is the inductive limit over all finite sums of the spaces $C_c^\infty(gH)$.
Note that the definition is independent of the choice of $H$, since, given a second open group $H'$, the support of any given $f\in C_c(G)$ will only meet finitely many left cosets $gH''$ of the open subgroup $H''=H\cap H'$.
It follows in particular, that $C_c^\infty(G)$ is the inductive limit over a family of Fr\'echet spaces.
This concludes the definition of the space $C_c^\infty(G)$ of test functions.

\begin{remark}
\begin{enumerate}[\rm (a)]
\item Note that the inductive limit topology in the category of locally convex spaces differs from the inductive limit topology in the category of topological spaces, as is made clear in \cite{Gloeck}.
\item Note that for a linear functional $\al:C_c^\infty(G)\to\C$ to be continuous, it suffices, that for any 
local Lie quotient $L$ of $G$ and any compact subset $K\subset L$ and any sequence $f_n\in C_K^\infty(L)$ with $f_n\to 0$ in the Fr\'echet space $C_K^\infty(L)$ and every $g\in G$ the sequence $\al(L_gf_n)$ tends to zero, where $L_gf(x)=f(g^{-1}x)$.
This is deduced from \cite{Schaef}, Chap II, Sec. 6.1.
\item If a locally compact group $G$ is a projective limit of Lie groups $G_j=G/N_j$, then it follows from \cite{Neeb}*{Cor. 12.3}, that every irreducible continuous representation $\pi$ factors through some $G_j$. This reduces many issues related to distribution characters to the case of Lie groups.
\end{enumerate}
\end{remark}
\end{definition}

\begin{definition}
A unitary representation $\pi$ of a locally compact group $G$ is called a \e{trace class representation}, if every operator $\pi(f)$ for $f\in C_c^\infty(G)$ is trace class.

We say that a locally compact group $G$ is a \e{trace class group}, if every irreducible unitary representation $\pi$ is trace class.
\end{definition}

\begin{proposition}\label{lem1.3}
Let $G$ be a locally compact group.
\begin{enumerate}[\quad\rm (a)]
\item Let $F$ be a normed $\C$-vector space and let $T:C_c^\infty(G)\to F$ be a linear map which is the point-wise limit of a  net of continuous linear maps $T_\al\to T$ such that for each $f\in C_c^\infty(G)$ one has
$$
\sup_\al\norm{T_\al(f)}<\infty.
$$
Then $T$ is continuous. 

Likewise, a semi-norm $N$ on $C_c^\infty(G)$, which is the supremum of a family of continuous semi-norms, is continuous.
\item  Let $\pi$ a trace class representation of $G$.
Then the ensuing linear functional
\begin{align*}
C_c^\infty(G)&\to\C,\qquad f\mapsto\tr\pi(f)
\end{align*}
is continuous.

Likewise, the map $f\mapsto \norm{\pi(f)}_{HS}$ mapping $f$ to the Hilbert-Schmidt norm of $\pi(f)$ is continuous.
\end{enumerate}
\end{proposition}

\begin{proof}
(a) 
Let $T_\al\to T$ be a point-wise convergent net, where each $T_\al$ is continuous.
Let $L$ be a local Lie quotient of $G$, let $K\subset L$ be a compact subset and let $g\in G$.
We can consider the space $C_K^\infty(L)$ as a subspace of $C_c^\infty(G)$.
By the remark above, it suffices to show that $T$ is continuous on $L_gC_K^\infty(L)$, or $T\circ L_g$ is continuous on $C_K^\infty(L)$.
As for each $f\in C_K^\infty(L)$ we have
$
\sup_\al|T_\al(f)|<\infty,
$
the Banach-Steinhaus-Theorem, (12.16.4) in \cite{Dieu} implies that the set of all $T_\al$ is equicontinuous on the Fr\'echet space $C_K^\infty(L)$.
By the uniform boundedness principle, see Chap III, 4.3 in \cite{Schaef} it follows that $T$ is continuous.

Next let $N=\sup_\al N_\al$ be a semi-norm, where the $N_\al$ are continuous semi-norms.
Then by \cite{Dieu} (12.7.7) the map $N$ is lower semicontinuous and by \cite{Dieu}(12.16.3) the map $N$ is continuous on $C_c^\infty(H)$ for each local Lie quotient $H$.
But a semi-norm on a locally-convex inductive limit is continuous if and only if it is continuous on each stage, see \cite{Bier}, Section 0.

(b) Let $(\pi,V_\pi)$ be a trace class representation.
Let $(e_i)_{i\in I}$ be an orthonormal basis of $V_\pi$.
For each finite subset $E\subset I$ let 
$T_E:C_c^\infty(G)\to \C$ be defined by
$$
T_E(f)=\sum_{i\in E}\sp{\pi(f)e_i,e_i}.
$$
Then for each $f\in C_c^\infty(G)$ we have
$$
\tr\pi(f)=\lim_{E\to I} T_E(f).
$$
So $f\mapsto \tr\pi(f)$ is the point-wise limit of the continuous linear maps $T_E$, further we have for each $f\in C_c^\infty(G)$  that
$$
\sup_E|T_E(f)|\le\sup_E\sum_{i\in E}|\sp{\pi(f)e_i,e_i}|\le \sum_i|\sp{\pi(f)e_i,e_i}|<\infty,
$$
as $\pi(f)$ is trace class.
By part (a) it follows that $f\mapsto\tr\pi(f)$ is continuous on $C_c^\infty(G)$.
The last statement on the Hilbert-Schmidt norm follows similarly, as $\norm{\pi(f)}_{HS}=\sup_E\(\sum_{i\in E}\norm{\pi(f)e_i}^2\)^{\frac12}$.
\end{proof}

\begin{examples}\label{Ex1}
\item Trivial examples of trace class groups include abelian and compact groups as their irreducible representations are finite-dimensional.
\item A representation $\pi$ of a discrete group $\Ga$ is trace class if and only if it is finite-dimensional.
To see the non-trivial direction, note that the function $f=\1_{\{1\}}$ is a test function, so for a trace class representation $\pi$, the operator $\pi(f)=\Id_\pi$ has to be trace class, which means that $\pi$ is finite-dimensional.
\end{examples}

For a topological group $G$ we denote by $\what G$ the set of isomorphism classes of irreducible unitary representations of $G$.

\begin{proposition}\label{prop1.6}
Let $G$ be a locally compact group and let $\pi\in\what G$.
Suppose that for every test function $f$ the operator $\pi(f)$ is Hilbert-Schmidt. Then $\pi(f)$ is trace class for every $f\in C_c^\infty(G)$.
\end{proposition}

\begin{proof}[Proof]
As the product of two Hilbert-Schmidt operators is trace class, this assertion follows from 
$$
C_c^\infty(G)*C_c^\infty(G)=C_c^\infty(G),
$$
where the left hand side expression stands for the space of all sums of the form $\sum_{j=1}^nf_j*g_j$, where $f_j,g_j\in C_c^\infty(G)$.
This latter assertion has been shown for Lie groups by Dixmier and Malliavin in \cite{DM}.
From there it follows easily by the definition of $C_c^\infty(G)$.
\end{proof}

A topological space $X$ is called \e{separable} if it contains a countable dense subset.
Recall that a topological group $G$ is called \e{type I}, if every factor representations is a multiple of an irreducible one.

The following theorem generalizes the claims in \cite{Funa}.

\begin{theorem}\label{thm1.7}
Any trace class group is of type I.
\end{theorem}

\begin{proof}
Let $G$ be a trace class group and let $(\pi,V_\pi)$ be an irreducible unitary representation.
Then for every $f\in C_c^\infty(G)$ the operator $\pi(f)$ is trace class and hence compact.
Since $\pi:L^1(G)\to \CB(V_\pi)$ is continuous and $C_c^\infty(G)$ is dense in $L^1(G)$, it follows that all operators $\pi(f)$ for $f\in L^1(G)$ are compact.
Since \cite{Dix}*{Section 13.9} the group $C^*$-algebra $C^*(G)$ is the completion of $L^1(G)$ with respect to the norm $\norm f_*=\sup_{\pi\in\what G}\norm{\pi(f)}_\op$, it follows that $\pi(f)$ is compact for every $f\in C^*(G)$ and every irreducible $*$-representation of $C^*(G)$.
This means that the group $G$ is liminal, hence postliminal [loc.cit.], i.e. of type I. 
\end{proof}

\begin{lemma}\label{lem1.8}
Let $X$ be a locally compact Hausdorff space equipped with a $\sigma$-finite Radon measure $dx$.
Let $\CA$ be a *-subalgebra of $\CB(L^2(X))$.
Assume that each $T\in\CA$ has a continuous kernel $k_T\in C(X\times X)$, such that $k_{TS}=k_T*k_S$ and $k_{S^*}=k_S^*$ hold pointwise. Assume further that $\CA=\CA^2$, where $\CA^2$ is the space of finite sums $\sum_{j=1}^nT_jS_j$ with $T_j,S_j\in\CA$.
Then the following are equivalent
\begin{enumerate}[\quad\rm (a)]
\item Each $T\in\CA$ is trace class.
\item For each $T\in\CA$ the integral 
$$
\int_X k(x,x)\,dx
$$
converges.
\end{enumerate}
If this is the case, one has $\tr(T)=\int_X k(x,x)\,dx$
for every $T\in\CA$.
\end{lemma}

\begin{proof}
(a)$\Rightarrow$(b) It suffices to assume that $T=S_1S_2$ and replacing $S_2$ with $S_2^*$ we can assume that $T=S_1S_2^*$.
The map $\sigma:(S_1,S_2)\mapsto S_1S_2^*$ is sesquilinear, so it obeys the polarization rule
$$
\sigma(S,R)=\frac14\(\sigma(S+R)-\sigma(S-R)+i(\sigma(S+iR)-\sigma(S-iR))\),
$$
where we have written $\sigma(S+R)$ for the diagonal $\sigma(S+R,S+R)$.
This implies that in order to show the lemma, it suffices to assume $T=SS^*$ for some $S\in\CA$.
As $k_{S^*}(x,y)=k_S(y,x)^*$ we conclude that
\begin{align*}
\tr(T)&=\tr(SS^*)=\norm S_{HS}^2=\int_X\int_X
\norm{k_S(x,y)}_{HS}^2\,dx\,dy\\
&=\int_X\int_X \tr (k_S(x,y)k_{S^*}(y,x))\,dy\,dx\\
&=\int_X\tr k_T(x,x)\,dx.
\end{align*}
(b)$\Rightarrow$(a) The calculation above, read backwards, yields that each $T\in\CA$ is Hilbert-Schmidt.
Now the fact that $\CA=\CA^2$ shows that each $T\in\CA$ is trace class.
\end{proof}

\begin{proposition}\label{prop1.7}
\begin{enumerate}[\quad\rm (a)]
\item Every connected nilpotent Lie group is trace class.
\item For $n\ge 2$ the group $\SL_n(\R)\ltimes\R^n$ is not trace class.
Therefore, there exists a connected unimodular Lie group of type I, which is not trace class.
\end{enumerate}
\end{proposition}

\begin{proof}
(a) This is due to Kirillov, \cite{Kiri}, see also \cite{Baggett}.

(b) The group $G=\SL_n(\R)\ltimes\R^n=Q\ltimes N$  is a connected algebraic group over the reals 
and  those have been shown to be type I by Dixmier in \cite{DixAlgLieGrps}.

Assume that $n\ge 2$  and that the group $G=\SL_n(\R)\ltimes\R^n$ is trace class.
Consider the Hilbert space $V=L^2(\R^n)$ and define the unitary representation $R$ of $G$ on $V$ by
$$
R(q,y)\phi(x)=\phi(qx
+y),\qquad q\in\SL_n(\R),\ x,y\in\R^n.
$$ 
By taking Fourier transforms one shows that $R$ is irreducible.
Let $f\in C_c^\infty(G)$.
Then for $\phi\in V$ we have
\begin{align*}
R(f)\phi(x)&=\int_Gf(y)R(y)\phi(x)\,dx\\
&=\int_Q\int_{\R^n}f(q,n)\phi(qx+n)\,dn\,dq\\
&=\int_Q\int_{\R^n}f(q,n-qx)\phi(n)\,dn\,dq.
\end{align*}
Therefore $R(f)$ is an integral operator on $L^2(\R^n)$ with continuous kernel
$$
k_f(x,y)=\int_Qf(q,y-qx)\,dq.
$$
 A computation shows that $k_{f*g}=k_f*k_g$.
By the Theorem of Dixmier-Malliavin the algebra $R(C_c^\infty(G))$ satisfies $\CA=\CA^2$ and so by Lemma \ref{lem1.8} we conclude that
$\tr(R(f))=\int_{\R^n}k_f(x,x)\,dx$ and in particular, the integral always converges.
Let now $f\in C_c^\infty(G)$ be of the form $f(q,n)=g(q)h(n)$ for some $g\in C_c^\infty(Q)$ and $h\in C_c^\infty(N)$ with $g,h\ge 0$ then by positivity we are allowed to interchange the order of integration in
\begin{align*}
\int_{\R^n}k_f(x,x)\,dx
&= \int_{\R^n}\int_Qf(q,x-qx)\,dq\,dx\\
&= \int_Q\int_{\R^n}f(q,(I-q)x)\,dx\,dq\\
&= \int_Q\int_{\R^n}|\det(I-q)|^{-1}f(q,x)\,dx\,dq\\
&= \int_Q|\det(I-q)|^{-1}g(q)\,dq\ \int_{\R^n}h(x)\,dx.
\end{align*}
But the integral $\int_Q|\det(I-q)|^{-1}g(q)\,dq$ will for $0\le g\in C_c^\infty(Q)$ not generally be finite.
So the assumption is false and so $G$ is not trace class.
\end{proof}

\begin{proposition}\label{prop1.9}
\begin{enumerate}[\quad\rm (a)]
\item If $G$ and $H$ are trace class groups, then so is their direct product $G\times H$.

\item If $F$ is an open subgroup of $G$ of finite index, then
\begin{center}
$G$ is trace class $\quad\Leftrightarrow\quad$ $F$ is trace class.
\end{center} 
\end{enumerate}
\end{proposition}

\begin{proof}
A proof of this proposition can be drawn from the results of \cite{Mili}. For the convenience of the reader we give an independent proof.

(a) As $G$ and $H$ are type I groups, any irreducible unitary representation $\tau$ of $G\times H$ is a Hilbert space tensor product $\tau=\pi\otimes \eta$ of representations $\pi\in\what G$ and $\eta\in \what H$. 
Let $g\in C_c^\infty(G\times H)$.
It suffices to show that $\tau(g)$ is Hilbert-Schmidt, i.e., $\infty>\tr(\tau(g)^*\tau(g))=\tr(\tau(g^**g))$.
So let $f=g^**g$, where $g^*(x)=\Delta(x^{-1})\ol{g(x^{-1})}$.
Fix orthonormal bases $(v_i)_{i\in I}$ of the representation space $V_\pi$ and $(w_j)_{j\in J}$ of $V_\eta$.
We have to show that $\sum_{i,j}\sp{\tau(f)v_i\otimes w_j,v_i\otimes w_j}<\infty$, where we note that each summand in this sum is $\ge 0$.
For every fixed $y\in H$, the function $f(\cdot,y)$ is in $C_c^\infty(G)$, therefore $\pi(f(\cdot,y))$ is trace class.
As the trace is continuous, the map $y\mapsto \tr\pi(f(\cdot,y))$ is continuous on $H$.
As $f$ factors over a Lie quotient of $G\times H$, we may assume that $G$ and $H$ are Lie groups for the moment.
For $X\in\Lie(H)$, the interchange of differentiation and integration shows that $\pi(Xf(\cdot,y))=X\pi(f(\cdot,y))$ and the same after applying the trace, so that after iteration we infer that the function $h:y\mapsto \tr\pi(f(\cdot,y))$ lies in $C_c^\infty(H)$.
Now we may abandon the condition that $G$ and $H$ be Lie groups again while still $h\in C_c^\infty(H)$.
Therefore the following trace exists in $\C$:
\begin{align*}
\tr\eta(h)&=\tr\(\int_H\tr\pi(f(\cdot,y))\eta(y)\,dy\)\\
&=\tr\(\int_H\pi(f(\cdot,y))\otimes\eta(y)\,dy\)\\
&=\sum_i\sum_j\sp{\(\int_H\pi(f(\cdot,y))\eta(y)\,dy\) v_i\otimes w_j,v_i\otimes w_j}\\
&=\sum_i\sum_j\sp{\tau(f) v_i\otimes w_j,v_i\otimes w_j}.
\end{align*}
Therefore the last double sum is indeed finite.

(b) 
Let $F$ be an open subgroup of $G$ of finite index.
Then $F$ contains an open subgroup $N$ which is normal in $G$ and has finite index, simply take $N$ to be the kernel of the homomorphism $G\to\mathrm{Per}(G/F)$.
Replacing the pair $(G,F)$ with $(G,N)$ and $(F,N)$  we see that it suffices to assume that $F=N$ is a normal subgroup of finite index.

So we first assume that $G$ is trace class and show that $N$ is.
Let $(\tau,V_\tau)\in\what N$ and let $\pi=\Ind_N^G(\tau)$ be the representation on the space $V_\pi$ of all maps $f:G\to V_\tau$ such that $f(ng)=\tau(n)f(g)$ holds for all $n\in N$ and $g\in G$.
Then the restriction of $\pi$ to $N$ is of finite length, hence, as a $G$-representation, $\pi$ is of finite length and as $G$ is trace class, so is $\pi$ and so is $\pi|_N$ as $N$ has finite index and is open in $G$.

The argument for the converse direction is due to the referee (we had a more complicated proof):
Assume that $N$ is trace class and let $\pi\in\what G$.
Let $\CA=\pi(N)^\circ$ be the commutant of $\pi(N)$.
This von Neumann algebra is acted upon by the finite group $F=G/N$ by conjugation. By the Lemma of Schur we know $\CA^F=\C\Id$.
Since $F$ is finite, by Proposition 2.1 of \cite{HKLS} it follows that $\CA$ is finite-dimensional.
Therefore $\pi|_N$ is of finite length and so, as $N$ is trace class, $\pi_N(f)$ is trace class for every $f\in C_c^\infty(N)$. Since $C_c^\infty(G)=\bigoplus_{g\in G/N}L_gC_c^\infty(N)$ it follows that $G$ is trace class.
\end{proof}

\begin{theorem}\label{thm2.1}
Suppose that $G$ is a unimodular group of type I and let $\mu$ be the Plancherel measure on $\what G$.
Let $f\in C_c^\infty(G)$, then $\pi(f)$ is of trace class for $\mu$-almost all $\pi\in\what G$.
\end{theorem}

\begin{proof}
Because of $C_c^\infty(G)=C_c^\infty(G)*C_c^\infty(G)$, it suffices to show that $\pi(f)$ is Hilbert-Schmidt $\mu$-almost everywhere.
This, however, is already true if $f\in L^1(G)\cap L^2(G)$ by the Plancherel Theorem.
\end{proof}

\section{Reductive and totally disconnected Groups}

By a \e{reductive Lie group} we mean a Lie group whose Lie algebra is reductive.

\begin{theorem}
A connected reductive Lie group is trace class.
\end{theorem}

\begin{proof}
Let $G$  be a connected reductive Lie group, let $G_\der$ be its derived group, $Z_\der$ the center of $G_\der$ and let $\ol G=G_\der/Z_\der$. 
Let $\ol K$ be a maximal compact subgroup of $\ol G$ and let $K$ be its pre-image in $G_\der$.
Now let $f\in C_c^\infty(G)$ and $(\pi,V_\pi)$ an irreducible unitary representation of $G$.
Then the center $Z$ of $G$ acts by a character $\chi_\pi:Z\to\T$.
Let $H=(G\times \T)/Z$, where $Z$ is embedded into the product by $z\mapsto (z,\chi_\pi(z)^{-1})$.
Then $\pi$ induces a representation $\pi_H$ of the group $H$ via $\pi_H(g,t)=t\pi(g)$.
The function $f\in C_c^\infty(G)$ induces a function $f_H\in C_c^\infty(H)$ by 
$$
f_H(g,t)=t^{-1}\int_Zf(gz)\chi_\pi(z)\,dz.
$$
We then have 
\begin{align*}
\pi_H(f_H)&=\int_Hf_H(x)\pi_H(x)\,dx\\
&=\int_{(G\times \T)/Z}t^{-1}\int_Zf(yz)\chi_\pi(z)\,dz\,t\pi(y)\,d[y,t]\\
&=\int_{G\times\T}f(y)\,\pi(y)\,dy\,dt=\pi(f).
\end{align*}
So it suffices to show that $\pi_H(f_H)$ is trace class, which means that, replacing $G$ with $H$, we have reduced to the case of the center of $G$ being compact.
In particular, then $K$ is compact and $G$ is a real reductive group in the sense of \cite{Wallach}.
We then consider the isotypical decomposition of $\pi|_K$, which we denote as $V_\pi=\bigoplus_{\tau\in\what K}V_\pi(\tau)$.
Then for each $\tau\in\what K$ one has $\dim V_\pi(\tau)\le (\dim\tau)^2$ and the Casimir operator $\Om_K$ of $K$ acts on $V_\pi(\tau)$ via the scalar $\tau(\Om_K)=-\norm{\la_\tau}^2-2\sp{\la_\tau,\rho}$, where $\la_\tau$ is the highest weight of $\tau$ and the norm derives from the negative Killing form $\sp{.,.}$, this formula can be derived from the Harish-Chandra isomorphism, see \cite{Knapp}*{Chap. VIII, Par. 5}.
The highest weight $\la_\tau$ is an element of the finitely generated abelian group of weights and by Weyl's character formula there exists a polynomial function $P$ such that $\dim\tau=P(\la_\tau)$, so in particular, it follows that there exists an integer $N$ such that $(\pi(\Om_K)-1)^{-N}$ is a trace class operator.
Now, as $f\in C_c^\infty(G)$, we get $(\Om_K-1)^Nf\in C_c^\infty(G)$ and so $\pi((\Om_K-1)^Nf)$ is a bounded operator.
Hence $\pi(f)=(\pi(\Om_K)-1)^{-N}\pi((\Om_K-1)^Nf)$ is a trace class operator.
\end{proof}

\begin{definition}
Let $G$ be a totally disconnected group.
A representation $(\pi,V_\pi)$ is called \e{admissible}, if $\dim V_\pi^K<\infty$ holds for every compact open subgroup $K$.

The group $G$ is called \e{admissible}, if every $(\pi,V_\pi)\in\what G$ is admissible.
\end{definition}

\begin{theorem}\label{thm4.3}
Let $G$ be a totally disconnected locally compact group and let $\pi\in\what G$.
Then
\begin{center}
$\pi$ is trace class \qquad $\Leftrightarrow$ \qquad $\pi$ is admissible.
\end{center}
In particular, a reductive linear algebraic group over a local field is trace class.\\
Further, the adelic points of a reductive linear algebraic group over a global field form a trace class group.

Also, let $\CN$ denote a unipotent algebraic group over a non-archimedean field $F$ of characteristic zero, then the group $N=\CN(F)$ is a trace class group.
\end{theorem}

\begin{proof}
For the first statement, $\pi$ be trace class and let $K\subset G$ be a compact open subgroup. Then $f=\1_K$ is in $C_c^\infty(G)$ and the operator
$\pi(f)$ equals the projection onto the $K$-invariants $V_\pi^K$.
As $\pi(f)$ is trace class, the latter space must be finite-dimensional. 

For the converse direction let $\pi$ be admissible and let $f\in C_c^\infty(G)$.
Then there exists a compact open subgroup $K\subset G$ such that $f$ factors over $K\bs G/K$, and so the operator $\pi(f)$ maps $V_\pi$ to a subspace of $V_\pi^K$.
The latter space is finite-dimensional, so $\pi(f)$ is of finite rank, hence trace class.

Harish-Chandra conjectured in \cite{Hari2}, and Bernstein proved in \cite{Bern}, that a reductive linear algebraic group over a local field is admissible, whence the second statement.
The last statement is shown in \cite{Dijk}, see the Remark at the end of Section 5.
\end{proof}

\begin{example}
Let $F$ be a non-archimedean local field and let $n\in\N$. We claim that the locally compact group $G=\GL_n(F)\ltimes F^n$ is not trace class.
For this equip $F^n$ with the additive Haar measure and consider the unitary representation $R$ on the space $V=L^2(F)$ given by
$$
R(g,v)\phi(x)=\sqrt{|\det g|^n}\phi(gx+v).
$$
As in Proposition \ref{prop1.7} one applies Fourier transform to see that this representation is irreducible.
Let $\CO$ be the ring of integers in $F$ and let $\pi\in\CO$ be a uniformizing element, i.e., a generator of the unique maximal ideal of $\CO$.
Then $K=\GL(\CO)\ltimes\CO^n$ is a compact open subgroup of $G$ and the space $V^K$ of $K$-invariants consists of all functions $\phi\in L^2(F)$ which are constant on the $K$-orbits in $F^n$.
Each $K$-orbit is compact and open in $F^n$, so, as $F^n$ is not compact itself, there are infinitely many of them.
Hence the space $V^K$ is infinite-dimensional, so the group $G$ is not admissible and hence not trace class.
\end{example}

\section{Discrete groups}

\begin{definition}
A locally compact group $G$ is \e{of finite representation type}, if every irreducible unitary representation of $G$ is finite-dimensional.
\end{definition}

\begin{examples}
\item Abelian groups and compact groups are of finite representation type.
\item If a locally compact group $G$ has a finite index closed abelian subgroup, then $G$ is of finite representation type.
\item Let $G$ be a locally compact group and let $Z$ be its center. If $G/Z$ is compact, then $G$ is of finite representation type.
\begin{proof}
Let $(\pi,V_\pi)$ be an irreducible unitary representation of $G$.
By the Lemma of Schur, the group $Z$ acts on $V_\pi$ through a character $\chi_\pi:Z\to\T$, where $\T=\{t\in\C:|t|=1\}$ is the circle group.
Embed $Z$ into the group $G\times\T$ via $z\mapsto (z,\chi_\pi(z)^{-1})$ and let $H=(G\times \T)/Z$.
Then $\pi$ induces an irreducible representation $\pi_H$ of the group $H$ via $\pi_H(g,t)=t\pi(g)$.
We finally claim that the group $H$ is compact, which then implies that $V_\pi$ is finite-dimensional.
As $K=G/Z$ is compact, there exists a compact set $C\subset G$ such that $G=CZ$.
Therefore, $H$ is the image of the compact set $C\times\T$ under the projection $G\times\T\to H$.
Hence $H$ is compact.
\end{proof}
\end{examples}

\begin{theorem}
For a discrete group $\Ga$ the following are equivalent:
\begin{enumerate}[\rm (a)]
\item $\Ga$ is trace class,
\item $\Ga$ is of finite representation type,
\item $\Ga$ is type I,
\item $\Ga$ is \e{abelian by finite}, i.e., there is an exact sequence
$$
1\to A\to\Ga\to F\to 1,
$$
where $A$ is abelian and $F$ is finite.
\end{enumerate}
\end{theorem}

\begin{proof}
The equivalence of (c) and (d) is in the paper \cite{Thoma}.
The equivalence of (a) and (b) has been shown in Examples \ref{Ex1}.
We now show (d)$\Rightarrow$(b).
So let $\pi\in \what\Ga$. The restriction to $A$ must be a direct integral over 
 and this integral is extended over one $F$-orbit in $\what A$ only. As $F$ is finite, so is the orbit, to $\pi|_A$ is the direct sum of finitely many one-dimensional representations, hence $\pi$ is finite-dimensional.

Finally, if (b) holds, then for every $\pi\in\what\Ga$, the set $\pi(\Ga)$ generates a finite dimensional factor von Neumann algebra $\CA(\pi)$. As every finite-dimensional factor von Neumann algebra is of type I, the algebra $\CA(\pi)$ is of type I, and as $\pi$ was arbitrary, the group $\Ga$ is of type I.
\end{proof}

\section{Semi-direct products}
Let $G=A\rtimes Q$ be a semi-direct product of locally compact groups, where $A$ is abelian.
Assume that $G$ is unimodular and that the quotient map $\what A\to \what A/Q$ has a Borel section.
The latter condition is known as \e{Mackey-regularity} 
\cites{Mackey,Varadarajan}.
It implies that for each $\pi\in\what G$ the representation $\pi|_A$ is a direct integral over a single $Q$-orbit $\CO_\pi$ in $\what A$.
A counterexample to Mackey-regularity is the Mautner group \cite{Baggett2}.

\begin{theorem}
Let $G=N\rtimes Q$ be a Mackey-regular semi-direct product of locally compact groups, with $N$ being abelian and $Q$ metrizable.
Let $Q_0$ denote the pointwise stabilizer of $N$ in $Q$. Assume that $Q_0$ is unimodular, $Q/Q_0$ is compact and that for each $\chi\in\what N$ the stabilizer $Q_\chi$ in $Q$ is trace class.
Then $G$ is trace class.
\end{theorem}

\begin{proof}
Let $\pi\in\what G$.
By Mackey's theory there exists $\chi\in\what N$ and an irreducible representation $(\eta,V_\eta)$ of the stabilizer $G_\chi$, extending $\chi$, such that $\pi=\Ind_{G_\chi}^G(\eta)$.
The space $V_\pi$ of $\pi$ is the space of all measurable functions $\phi:G\to V_\eta$ such that $\phi(zx)=\eta(z)\phi(x)$ holds for all $z\in G_\chi$ and all $x\in G$, and $\int_{G_\chi\bs G}\norm{\phi(x)}^2\,dx<\infty$.
The representation $\pi$ is defined by right translations, $\pi(y)\phi(x)=\phi(xy)$.
As $Q_\chi$ is metrizable and cocompact in $Q$, there exists a measurable section $s:Q_\chi\bs Q\to Q$, whose image has compact closure see \cite{FG}, in particular Remark 6 (where one should read $KCG/M$ as $K\subset G/M$).
As $Q_0$ is unimodular, 
then so is $N\times Q_0$ and as $N\times Q_0$ is cocompact in $G$, any closed group $H$ with $N\times Q_0\subset H\subset G$ is unimodular by \cite{HA2}*{Proposition 9.1.2}.

For $f\in C_c^\infty(G)$ and $\phi\in V_{\pi}$ we compute
\begin{align*}
\pi(f)\phi(x)&=\int_Gf(y)\phi(xy)\,dy\\
&=\int_Gf(x^{-1}y)\phi(y)\,dy\\
&=\int_{G_\chi\bs G}\int_{G_\chi}f(x^{-1}g_\chi\bar g)\eta(g_\chi)\,dg_\chi\,\,\phi(s(\bar g))\,d\bar g\\
\end{align*}
Note that $G_\chi=N\rtimes Q_\chi$.
As $\pi(f)\phi$ is determined by its values on $s(Q_\chi\bs Q)$ we may assume $x=s(\bar x
)\in Q$.
Then 
$$
\pi(f)\phi(s(\bar x))
$$
equals
$$
\int_{Q_\chi\bs Q}\int_N\int_{Q_\chi} f(s(\bar x)^{-1}ns(\bar x),s(\bar x)^{-1}q_\chi\bar q)\eta(n,q_\chi)\,dq_\chi\,dn\ \phi(s(\bar q))\,d\bar q.
$$
We may therefore view $\pi(f)$ as an integral operator on the Hilbert space $L^2(Q_\chi\bs Q,V_\eta)$ with kernel
$$
k(\bar x,\bar q)= \int_N\int_{Q_\chi} f(s(\bar x)^{-1}ns(\bar x),s(\bar x)^{-1}q_\chi\bar q)\eta(n,q_\chi)\,dq_\chi\,dn
$$
As the dependence on $\bar x$ is through $s(\bar x)$ only, we may write $k(\bar x,\bar q)=\tilde k(s,\bar q)$, where $\tilde k$ is a map on $Q\times (Q_\chi\bs Q)$.
Since $f$ has compact support and the image of $s$ has compact closure, the integral over $N$ can be replaced with $\int_K$ for some compact subset $K\subset N$.
As $N$ acts on $V_\eta$ through the character $\chi$, the representation $\eta|_{Q_\chi}$ remains irreducible. Since $Q_\chi$ is trace class, for fixed $a\in N$ and $\bar x,\bar q\in Q_\chi\bs Q$ the integral over $Q_\chi$ defines a trace class operator $K_{s(\bar x),n,\bar q}(f)$ on $V_\eta$.
As the map $(s,n,\bar q)\mapsto \norm{K_{s,n,\bar q}(f)}_{HS}$ is continuous by Proposition \ref{lem1.3}, the map
$$
(s,\bar q)\mapsto \int_N\norm{K_{s,n,\bar q}(f)}_{HS}\,dn=\int_K\norm{K_{s,n,\bar q}(f)}_{HS}\,dn
$$
is continuous on $Q\times(Q_\chi\bs Q)$.
As the image of the map $s$ is relatively compact, the map $(\bar x,\bar q)\mapsto \int_N \norm{K_{s(\bar x),n,\bar q}(f)}_{HS}\,dn$ is bounded, by, say $C>0$.
It follows that
$$
\norm{\int_N K_{s(\bar x),n,\bar q}(f)\,dn}_{HS}\le \int_N \norm{K_{s(\bar x),n,\bar q}(f)}_{HS}\,dn\le C.
$$
Therefore, by the compactness of $Q_\chi\bs Q$,
$$
\int_{Q_\chi\bs Q}\int_{Q_\chi\bs Q}\norm{\int_N K_{s(\bar x),n,\bar q}(f)\,dn}_{HS}^2\,d\bar x\,d\bar q
\le 
\int_{Q_\chi\bs Q}\int_{Q_\chi\bs Q}C^2\,d\bar x\,d\bar q<\infty.
$$
This double integral is the square of the Hilbert-Schmidt norm of the operator $\pi(f)$ which therefore is finite, hence $\pi(f)$ is a Hilbert-Schmidt operator and as $f$ was arbitrary, $G$ is trace class.
\end{proof}

In the case of a semi-direct product of abelian groups we can give an `if and only if' criterion for trace class.
The following proposition was part of the thesis of F.J.M. Klamer (Groningen 1979), which unfortunately has never been published.

Let $G=N\rtimes Q$, where $N$ and $Q$ are abelian.
Assume that the Haar measure of $N$ is $Q$-invariant, or, which amounts to the same, that $G$ is unimodular. 
Let $\chi\in\what N$ and let $Q_\chi$ be its stabilizer in $Q$. Select $\rho\in\what Q_\chi$ and define the representation $\pi=\pi_{\chi,\rho}$ as 
$$
\pi_{\chi,\rho}=\Ind_{N\rtimes Q_\chi}^G\chi\otimes\rho.
$$
Then $\pi_{\chi,\rho}$ is irreducible.

\begin{theorem}
The representation $\pi_{\chi,\rho}$ is trace class if and only if the $Q$-invariant Radon measure $\mu$ on the $Q$-orbit $\CO_\chi\subset\what N$ of $\chi$, which is unique up to scaling, yields a tempered measure  on the dual group $\what N$, i.e., $\int_{\what N}\phi\,d\mu$ exists for every Schwartz-Bruhat function $\phi$ on $\what N$.
\end{theorem}

This theorem yields new interesting classes of counterexamples like the group
$$
\left\{\mat {e^t}\ \ {e^{-t}}: t\in\R\right\}\ltimes \R^2,
$$
which by the theorem is not trace class.

\begin{proof}
As in the proof of the last theorem one sees that for $f\in C_c^\infty(G)$ the operator $\pi(f)$ has kernel
\begin{align*}
k_f(x,y)&= \int_N\int_{Q_\chi} f( x^{-1}nx, x^{-1}q_\chi y)\chi(n)\rho(q_\chi)\,dq_\chi\,dn\\
&=\int_{Q_\chi}\int_N f(n, x^{-1}q_\chi y)\chi(xnx^{-1})\,dn\ \rho(q_\chi)\,dq_\chi,\\
&=\int_{Q_\chi}\int_N f(n, x^{-1}q_\chi y)\chi^x(n)\,dn\ \rho(q_\chi)\,dq_\chi.\\
\end{align*}
Let $h_q(n)=f(n,q)$, then, as we integrate over $N$, for $x=y$ we get
$$
k_f(x,x)
=\int_{Q_\chi}\hat h_{q_\chi}(\chi^x) \rho(q_\chi)\,dq_\chi.
$$
The map $q\mapsto \hat h_q$ is a continuous map of compact support from $Q$ to $\CS(\what N)$.
So if the measure $\mu$ extends as in the theorem, then the map $q\mapsto \int_{Q_\chi\bs Q}\hat h_q(\chi^x)\,dx$ is continuous of compact support, hence the integral 
$$
\int_{Q_\chi\bs Q}k_f(x,x)\,dx=\int_{Q_\chi}\int_{Q_\chi\bs Q}\hat h_{q_\chi}(\chi^x)\rho(q_\chi)\,dx\,dq_\chi
$$
exists, so that by Lemma \ref{lem1.8} it follws that $\pi_{\chi,\rho}$ is trace class.

For the converse direction assume that $\pi_{\chi,\rho}$ is trace class.
Consider the special case $f=\phi\otimes\psi$ by which we mean $f(n,q)=\phi(n)\psi(q)$ with $\phi\in C_c^\infty(N)$ and $\psi\in C_c^\infty(Q)$.
Fix $\psi$ and let $T_\psi(\phi)=\tr\pi(\phi\otimes\psi)$.
Then $h_q(n)=\phi(n)\psi(q)$ and hence
$$
T_\psi(\phi)=\underbrace{\int_{Q_\chi}\psi(q_\chi)\rho(q_\chi)\,dq_\chi}_{=C_\psi} \int_{Q_\chi\bs Q}\hat\phi(\chi^x)\,dx=C_\psi\int_{\CO_\chi}\hat\phi(z)\,d\mu(z).
$$
Note that $\psi$ can be chosen such that the constant $C_\psi$ is non-zero.
So, up to a constant, $T_\psi$ is the Fourier transform of the measure $\mu$.
Since $(n,q)^{-1}=(n^{-q},q^{-1})$ we compute
\begin{align*}
f*f^*(n,q)
&= \int_Gf(n_1,q_1)f^*((n_1,q_1)^{-1}(n,q))\,dn_1\,dq_1\\
&= \int_Gf(n_1,q_1)\ol{f((n,q)^{-1}(n_1,q_1))}\,dn_1\,dq_1\\
&= \int_Gf(n_1,q_1)\ol{f((n^{-q},q^{-1})(n_1,q_1))}\,dn_1\,dq_1\\
&= \int_Gf(n_1,q_1)\ol{f(n^{-q}n_1^q,q^{-1}q_1)}\,dn_1\,dq_1=\phi*\phi^*(n^q)\,\psi*\psi^*(q).
\end{align*}
As $\tr\pi(f*f^*)\ge 0$ we conclude that the distribution $T_{\psi*\psi^*}$ is positive definite. Hence it is tempered \cite{Schwartz}*{Theoreme XVIII, p. 132}.
By the same theorem, $T_{\psi*\psi^*}$ is the Fourier-transform of a tempered Radon-measure on $\what N$.
Using $C_c^\infty(N)=C_c^\infty(N)*C_c^\infty(N)$ and polarizing, one sees that the same assertions follow for $T_\psi$.
But by the formula above, $T_{\psi}$ is $C_\psi$ times the Fourier-transform of the measure $\mu$, which therefore is tempered. 
\end{proof}

{\bf Remark.}
There remain many open questions. For instance, is every trace class group unimodular? Is it true that a group $G$ is of trace class if and only if it admits a uniform lattice?
Let $G$ be a connected Lie group. Is it true that $G$ is trace class if and only if $G_{red}$ acts with closed orbits on the dual $\what N$ of the unipotent radical $N$?
Is a unipotent algebraic group over a local field trace class?
(The case of characteristic zero has been dealt with, see Proposition \ref{prop1.7} and Theorem \ref{thm4.3}.)

\section{Distributional realizations}
Let $G$ be a locally compact group and let $C_c^\infty(G)'$ denote the space of all continuous linear forms on $C_c^\infty(G)$.
We equip $C_c^\infty(G)'$ with the topology of point-wise convergence.
The group $G$ acts continuously on $C_c^\infty(G)'$ by left- and right translations and these two commute with each other, so the group $G\times G$ acts continuously on $C_c^\infty(G)'$.
We say that a unitary representation $(\alpha,V_\alpha)$ of $G\times G$ can be realized as a \e{$G\times G$-stable subspace} of $\CE=C_c^\infty(G)'$, if there exists an injective continuous linear $G\times G$-map $j:V_\alpha\to \CE$.

If $\pi$ and $\eta$ are irreducible unitary representations of $G$, then $\pi
\otimes\eta$ is an irreducible unitary representation of $G\times G$ on the Hilbert space completion $V_\pi\hat\otimes V_\eta$ of the algebraic tensor product $V_\pi\otimes V_\eta$.
If $G$ is type I, then every irreducible unitary representation of $G\times G$ is such a tensor product with uniquely determined factors.

The following theorem is due to F.J.M. Klamer (thesis, University of Groningen 1979).
Unfortunately, it has never been published.

\begin{theorem}\label{thm3.1}
Let $G$ be a  locally compact group. Then the irreducible unitary representation $\pi\otimes\eta$ of $G\times G$ can be realized as a $G\times G$ stable subspace of $\CE$ if and only if $\eta$ is isomorphic to the dual representation $\pi'$ of $\pi$ and $\pi$ is a trace class representation.

In particular, a type I group $G$ is trace class if and only if for every $\pi\in\what G$ the representation $\pi\hat\otimes\pi'$ can be realized as a subspace of $\CE$.
\end{theorem}

Note that the topology of $V_\pi\hat\otimes V_{\pi'}$ as a subspace of $\CE$ does in general not coincide with the Hilbert space topology.

The proof of the theorem will be given in a sequence of lemmas.

\begin{lemma}\label{lem3.2}
If $G$ is a locally compact group and $\pi\in\what G$ is a trace class representation, then $\pi\hat\otimes\pi'$ can be realized as a stable subspace of $\CE$.
\end{lemma}

\begin{proof}
Let $(e_i)$ denote an orthonormal basis of $V_\pi$ and $(e_i')$ the dual basis of $V_{\pi'}$.
Denote the duality pairing between $V_\pi$ and $V_{\pi'}$ by $\sp{v|\al}=\al(v)$.
We define a map $j:V_\pi\otimes V_{\pi'}\to C(G)\subset\CE$ by
$$
j(v\otimes\al)(x)=\sp{v|\pi'(x)\al}.
$$
Then $j$ is linear and $G\times G$ equivariant in the sense that
$$
j(\pi(x)\otimes\pi'(y)(v\otimes\al))=L_xR_yj(v\otimes\al).
$$
For $f\in C_c^\infty(G)$ one has
\begin{align*}
j\(\sum_{i,j}\la_{i,j}e_i\otimes e_j'\)(f)
&= \sum_{i,j}\la_{i,j}\int_G{f(x)}\sp{e_i|\pi'(x) e_j'} \,dx\\
&=\sum_{i,j}\la_{i,j}\sp{e_i|\pi'(f)e_j'}.
\end{align*}
As $\pi(f)$, and hence $\pi'(f)$, is trace class, it follows from Lemma 5.3.5 of \cite{HA2}, that $\sum_{i,j}\left|\sp{e_i| \pi'(f)e_j'}\right|<\infty$ and so $\sum_{i,j}\left|\sp{e_i|\pi'(f) e_j'}\right|^2<\infty$ and by the Cauchy-Schwarz inequality we get 
$$
\sum_{i,j}|\la_{i,j}\sp{e_i|\pi'(f) e_j'}|\le \(\sum_{i,j}|\la_{i,j}|^2\)^{\frac12}\(\sum_{i,j}\left|\sp{e_i|\pi'(f) e_j'}\right|^2\)^{\frac12},
$$
so that $j$ extends to a continuous linear map $V_\pi\hat\otimes V_{\pi'}\to\CE$ which
must be injective, as $j\ne 0$ and $\pi\otimes\pi'$ is irreducible.
Lemma \ref{lem3.2} is proven.
\end{proof}

The addendum of the theorem becomes clear from this: The topology induced by the map $j$ on $V_\pi\hat\otimes V_{\pi'}$ is not the Hilbert topology if $V_\pi$ is infinite dimensional. For this let $(i_k)_{k\in\N}$ be a sequence of pairwise distinct indices. Then the sequence $v_k=e_{i_k}\otimes e_{i_k}'$ does not converge in $V_\pi\hat\otimes V_{\pi'}$, but we show that $j(v_k)$ tends to zero in $\CE$.
This means we have to show that for  given $f\in C_c^\infty(G)$ the sequence $\sp{e_{i_k}|\pi(f)e_{i_k}'}$ tends to zero. This, however is clear as $\pi(f)$ is trace class.

For the rest of the proof of Theorem \ref{thm3.1}, we first reduce to the case of a unimodular group.
So let $G$ be  locally compact and let $\Delta:G\to (0,\infty)$ be its modular function.
Let $H=G\ltimes\R$, where $G$ acts on $\R$ through $\Delta$, so as a set we have $H=G\times \R$ and the multiplication is $(g,x)(h,y)=(gh,\Delta(h^{-1})x+y)$.
Then $H$ is unimodular and the projection onto the first factor yields a surjective group homomorphism $\al:H\twoheadrightarrow G$.
So let $j:V_\pi\hat\otimes V_\eta\to C_c^\infty(G)'$. Integrating over $\R$ yields a continuous linear map $\al_*:C_c^\infty(H)\to C_c^\infty(G)$, which dualizes to $\al':C_c^\infty(G)'\to C_c^\infty(H)'$.
We can consider $\pi$ and $\eta$ as representations of $H$ and post-composing with $\al'$ we get a $H\times H$ map $j: V_\pi\hat\otimes V_\eta\to C_c^\infty(H)'$.
If we assume the theorem proven for the unimodular group $H$, we infer $\eta\cong\pi'$ and $\pi(f)$ trace class for every $f\in C_c^\infty(G)$.

So for the rest of the proof of Theorem \ref{thm3.1} we can assume $G$ to be unimodular.
For $f\in C_c^\infty(G)$ we define $f^*(x)=\ol{f(x^{-1})}$.
Then  for any unitary representation $\pi$ of $G$ we have 
$\pi(f^*)=\pi(f)^*$.

\begin{lemma}
\begin{enumerate}[\quad\rm (a)]
\item Let $G$ be a locally compact group and let $I:C_c^\infty(G)\to\C$ be a continuous, left-invariant linear functional, i.e., $I(L_yf)=I(f)$ holds for all $y\in G$, $f\in C_c^\infty(G)$.
Then there exists $c\in\C$ such that
$$
I(f)=c\int_Gf(x)\,dx.
$$
\item Let $G$ be unimodular and let $\tilde D:C_c^\infty(G)\otimes C_c^\infty(G)\to\C$ be a continuous linear map where the tensor product is equipped with the projective topology.
Assume that $\tilde D(R_yf\otimes L_yg)=\tilde D(f\otimes g)$ holds for all $f,g\in C_c^\infty(G)$.
Then $\tilde D$ factors over the convolution product, i.e., there exists a distribution $D$ on $G$ such that $\tilde D(f\otimes g)=D(f*g)$.
\end{enumerate}
\end{lemma}

\begin{proof}
(a) First note that for $F\in C_c^\infty(G\times G)$ the function $x\mapsto\int_GF(x,y)\,dy$ lies in $C_c^\infty(G)$ and one has $T_x\(\int_G F(x,y)\,dy\)=\int_G T_x(F(x,y))\,dy$.
For each open unit-neighborhood $U$ fix a function $\phi_U\in C_c^\infty(U)$ such that $\phi\ge 0$ and $\int_G\phi(x^{-1})\,dx=1$.
The set $\CU$ of all open unit-neighborhoods is a directed set with the inverse inclusion, i.e., $U\le V\ \Leftrightarrow\ U\supset V$.
Thus the map $U\mapsto \phi_U$ is a net in $C_c^\infty(G)$.
We claim that for every distribution $T$ and every $f\in C_c^\infty(G)$ the net of complex numbers $T(f*\phi_U)$ converges to $T(f)$.
For this let $\eps>0$, then by continuity of the right translation, there exists an open unit-neighborhood $U_0$ such that for $y\in U_0^{-1}$ one has $|T_x(f(xy)-f(x))|<\eps$.
Let $U\in\CU$ with $U\subset U_0$, then one has
\begin{align*}
|T(f*\phi_U)-T(f)|&=|T(f*\phi_U)-f)|\\
&=\left|T_x\(\int_Gf(xy)\phi_U(y^{-1})\,dy-\int_Gf(x)\phi(x^{-1})\,dx\)\right|\\
&=\left|T_x\(\int_{U^{-1}}(f(xy)-f(x))\phi_U(y^{-1})\,dy\)\right|\\
&=\left|\int_{U^{-1}}T_x(f(xy)-f(x))\phi_U(y^{-1})\,dy\right|\\
&\le\int_{U^{-1}}\underbrace{|T_x(f(xy)-f(x))|}_{<\eps}\phi_U(y^{-1})\,dy<\eps.
\end{align*}
Now let $I$ be a left-invariant distribution, then for $f,\phi\in C_c^\infty(G)$ one computes
\begin{align*}
I(f*\phi)&=I_x\(\int_Gf(y)\phi(y^{-1}x)dy\)\\
&= \int_Gf(y) I_x(\phi(y^{-1}x))dy\\
&= \int_Gf(y) I_x(\phi(x))dy= \int_Gf(y)dy\  I(\phi).
\end{align*}
If $\phi$ runs through the net $(\phi_U)_{U\in\CU}$ then the left hand side of this equation converges to $I(f)$.
Assuming $I\ne 0$ and choosing $f$ with $ 0\ne\int_Gf(x)\,dx$, we conclude that $I(\phi_U)$ converges to $c=\frac{I(f)}{\int_Gf(y)\,dy}$.
Finally, for any $f$,
$$
I(f)=\lim_UT(f*\phi_U)=\int_Gf(y)dy\lim_UI(\phi_U)=c\int_Gf(y)dy.
$$

(b)
The projective completion of the tensor product is $C_c^\infty(G\times G)$, so $\tilde D$ extends uniquely to a distribution on $G\times G$ which satisfies
$\tilde D(h(xa,a^{-1}y))=\tilde D(h(x,y))$ for every $a\in G$, or in other words $\tilde D\circ m_a^*=\tilde D$ for every $a\in G$, where $m_a:G\times G\to G\times G$, $(x,y)\mapsto (xa,a^{-1}y)$ and $m_a^*(f)=f\circ m_a$.

For a function $h$ on $G\times G$ we  write $L_ah(x,y)=h(x,a^{-1}y)$.
Consider the map $I:C_c^\infty(G\times G)\to C_c^\infty(G)$ given by $I(h)(x)=\int_Gh(x,y)\,dy$.
Let $T$ be a distribution on $G\times G$ which is invariant under left translations in the second argument, i.e., which satisfies $T(L_ah)=T(h)$ for all $a\in G$.
We claim that $T$ factors over $I$.
To see this, let $f,g\in C_c^\infty(G)$, then the function $f\otimes g(x,y)=f(x)g(y)$ lies in $C_c^\infty(G\times G)$.
For fixed $f$ the distribution $g\mapsto T(f\otimes g)$ is $G$ left-invariant, hence factors over the integral, i.e., there exists a number $D(f)\in\C$ such that $T(f\otimes g)=D(f)\int_Gg(y)\,dy$.
Fixing $g$ and varying $f$ one sees that the map $D:C_c^\infty(G)\to\C$ is linear and continuous and we have $T(h)=D(I(h))$ for every $h$ of the form $h=f\otimes g$ for some $f,g\in C_c^\infty(G)$. As the linear span of all functions of the form $f\otimes g$ is dense in $C_c^\infty(G\times G)$, we get $T(h)=D(I(h))$ for every test function  $h$ on $G\times G$.

Next consider the map $\mu:G\times G\to G\times G$, $(x,y)\mapsto (xy,y)$.
Then the inverse is $\mu^{-1}(x,y)=(xy^{-1},y)$, so $\mu$ is a homeomorphism. If  $f\in C_c^\infty(G\times G)$, then $\mu^*(f)=f\circ \mu$ again is a test function.
We have $\mu(m_a(x,y))=\mu((xa,a^{-1}y)=(xy,a^{-1}y)=l_a(xy,y)=l_a(\mu(x,y))$, where we have written $l_a(x,y)=(x,a^{-1}y)$. We then have $L_a=l_a^*$.
It follows that $m_a^*\mu^*=\mu^*L_a$.
Define $T=\tilde D\circ\mu^*$, then $TL_a=\tilde D\mu^*L_a=\tilde Dm_a^*\mu^*=\tilde D\mu^*=T$, so it follows that there exists a distribution $D$ such that $T=DI$ and hence
\begin{align*}
\tilde D(f\otimes g)&=T\mu^{-*}(f\otimes g)\\
&=D(I(\mu^{-*}(f\otimes g)))\\
&=D_x(\int_G\mu^{-*}(f\otimes g)(x,y)\,dy\\
&=D_x\(\int_G(f(xy^{-1})g(y)\,dy\)=D(f*g).\mqed
\end{align*}
\end{proof}

\begin{lemma}\label{lem3.3}
Let $G$ be a unimodular locally compact group and let $\pi,\eta\in\what G$.
Assume there exists a continuous linear $G\times G$-map $j:V_\pi\hat\otimes V_{\eta}\to\CE$.
\begin{enumerate}[\rm (a)]
\item
Then the prescription
$$
\sp{v,j^*\(\ol f\)}=\sp{f\mid jv}
$$
defines a continuous linear $G\times G$ map $j^*: C_c^\infty(G)\to V_{\pi}\hat\otimes V_{\eta}$.
Here $\ol f$ is the complex conjugate of $f$ and $\sp{.,.}$ denotes the inner product of the Hilbert space $V_\pi\hat\otimes V_\eta$, which we choose to be linear in the first slot.
\item There exists a distribution $D\in\CE$ such that
$$
D(f*g)=\sp{j^*(f),j^*(g^*)}.
$$
The distribution $D$ satisfies $D(f^*)=\ol{D(f)}$ and $D(f*g)=D(g*f)$.
The distribution $D$ is called the \e{reproducing distribution} of $j$.
\item One has $\eta\cong\pi'$.
\end{enumerate}\end{lemma}

The unimodularity of $G$ is necessary to assure $D(f*g)=D(g*f)$ which is crucial to the application in part (c).

\begin{proof}
The map $j^*$ is clearly well-defined, linear, $G\times G$-equivariant and continuous.
The image of $j^*$ is non-trivial and $G\times G$-stable, hence it is a dense subspace of $V_\pi\hat\otimes V_\eta$ as $\pi\otimes\eta$ is irreducible.
The map $(f, g)\mapsto \sp{j^*(f),j^*(g^*)}$ is a continuous bilinear map on $C_c^\infty(G)\times C_c^\infty(G)$, so it induces a continuous linear $\tilde D$ map on $C_c^\infty(G)\otimes C_c^\infty(G)$, where the latter is equipped with the projective topology.
Denote the left and right actions of $G$ on $C_c^\infty(G)$ by $L$ and $R$, so $L_yf(x)=f(y^{-1}x)$ and $R_yf(x)=f(xy)$. Then, as $G$ is unimodular, $(L_yg)^*=R_yg^*$ so that
$
\tilde D(R_yf\otimes L_yg)=\tilde D(f\otimes g)
$
holds for all $y\in G$.
By the last lemma,
this implies that $D$ factors over the convolution product, so  there exists a distribution $D\in \CE$ such that 
$\tilde D(f\otimes g)=D(f*g)$, i.e.,
$$
D(f*g)=\sp{j^*(f),j^*(g^*)}.
$$
Further, as also $\tilde D(L_y f\otimes R_yg)=\tilde D(f\otimes g)$ we conclude that the distribution $D$ is invariant under $G$-conjugation, which implies that
$$
D(f*g)=D(g*f)
$$
holds for all $f,g\in C_c^\infty(G)$.
Also, as 
$$
D\((f*g)^*\)=D(g^**f^*)=\sp{j^*(g^*),j^*(f)}=\ol{D(f*g)}
$$ 
we conclude that 
$$
D(f^*)=\ol{D(f)}
$$
holds for every $f\in C_c^\infty(G)$.

Now if $f\in C_c^\infty(G)$ lies in the kernel of $j^*$, then it follows that $D(f*g)=0$ for all $g\in C_c^\infty(G)$. This implies $D(g*f)=0$ and this implies that $j^*(f^*)=0$, so the kernel of $j^*$ is stable under $f\mapsto f^*$,
Therefore the map
\begin{align*}
\Theta: \Im(j^*)&\to \Im(j^*),\\
j^*(f)&\mapsto j^*(f^*)
\end{align*}
is well defined. It is anti-linear and by
\begin{align*}
\sp{\Theta(j^*(f)),\Theta(j^*(g))}
&= \sp{j^*(f^*),j^*(g^*)}\\
&= D(f^**g)=\ol{D(g^**f)}=\ol{D(f*g^*)}\\
&= \ol{\sp{j^*(f),j^*(g)}}
\end{align*}
it is an isometry, so it extends to an anti-linear isometry on the entire Hilbert space $V_\pi\hat\otimes V_\eta$.
Finally, for $x,y\in G$ we get
\begin{align*}
\Theta(\pi(x)\otimes\eta(y) j^*(f))
&= \Theta(j^*(L_xR_y f))\\
&= j^*(R_xL_y f^*)= \pi(y)\otimes\eta(x) \Theta(j^*(f)).
\end{align*}
This implies that $\Theta$ followed by the flip $v\otimes w\mapsto w\otimes v$ is an antilinear $G\times G$ isomorphism $V_\pi\hat\otimes V_\eta\tto\cong V_\eta\hat\otimes V_\pi$.
Followed by the antilinear isomorphism from a Hilbert space to its dual we get a $G\times G$ isomorphism
$$
V_\pi\hat\otimes V_\eta\tto\cong V_{\eta'}\hat\otimes V_{\pi'},
$$
and hence $\eta'\cong\pi$ and the proof of Lemma \ref{lem3.3} is finished.
\end{proof}

In order to proceed, we need traces on von Neumann algebras. 
Let $H$ be a Hilbert space and let $\CA\subset \CB(H)$ be a von Neumann algebra. As in \cite{HA2}, we write $\CA^\circ$ for its commutant.
Write $\CA^+$ for the set of elements of the form $T^*T$, where $T\in\CA$.
Note that $\CA^+$ is closed under addition, as for $a,b\in \CA$ the operator $T=a^*a+b^*b$ is positive, therefore has a positive square root $S=\sqrt T$, which lies in $\CA$, as $\CA$ is weakly closed \cite{HA2}.
Recall that a \e{trace} on $\CA$ is a map $\tau:\CA^+\to [0,\infty]$ subject to the following conditions
\begin{enumerate}[\rm (i)]
\item $\tau(S+T)=\tau(S)+\tau(T)$,
\item $\tau(\la S)=\la\tau(S)$,
\item $\tau(USU^{-1})=\tau(S)$
\end{enumerate}
for all $S,T\in\CA^+$, all $\la\ge 0$ and all unitary $U\in\CA$.
The trace is called \e{faithful}, if $\tau(S)=0$ implies $S=0$.
It is called \e{finite}, if $\tau(S)<\infty$ holds for every $S\in\CA^+$.
Further, $\tau$ is called \e{semi-finite}, if for every $0\ne S\in\CA^+$ there exists $T\in\CA^+$ with $T\le S$ and $0<\tau(T)<\infty$.
Here $T\le S$ means $S-T\in\CA^+$. This establishes a partial order on $\CA$.
The trace $\tau$ is called \e{normal}, if for every subset $\CF\subset\CA^+$ with $\sup\CF=S\in\CA$ one has $\tau(S)=\sup_{T\in\CF}\tau(T)$.

Recall that a von Neumann algebra is called a \e{factor}, if its center equals $\C \Id$.
In \cite{DixvN} 6.4, Corollary to Theorem 3, it is shown that on a factor any two  faithful, semi-finite normal traces are proportional.
For instance on $\CA=\CB(H)$ the standard trace
$\tr$, with
$$
\tr(T)=\sum_j\sp{Te_j,e_j}
$$
for any orthonormal basis $(e_j)$ is a faithful semi-finite normal trace and any other therefore is a positive multiple of $\tr$.

For any algebra $\CA$ write $\CA^2$ for the span of all elements of the form $ab$, with $a,b\in \CA$.

\begin{definition}
A $*$-algebra $\CH_0$, which satisfies $\CH_0^2=\CH_0$, is called a \e{Hilbert algebra}, if it is endowed with an inner product $(.\mid.)$ such that for all $x,y,z\in\CH_0$ one has
\begin{enumerate}[\rm (i)]
\item $(x\mid y)=(y^*\mid x^*)$,
\item $(xy\mid z)=(x\mid zy^*)$,
\item the map $R_x:h\mapsto hx$ (and hence $L_x:h\mapsto xh$) is a bounded linear operator on $\CH_0$.
\end{enumerate}
\end{definition}

Let $\CH_0$ be a Hilbert algebra and denote by $\CH$ its Hilbert space completion.
For $x\in \CH_0$, the operators $R_x$ and $L_x$ extend to continuous operators on $\CH$.
Let $\CR,\CL\subset\CB(\CH)$ denote the von Neumann algebras generated by all $R_x$, $x\in\CH_0$ and $L_x$, $x\in\CH_0$ respectively.
Then by Theorem \cite{DixvN} I.5.2 we have
$$
\CR^\circ=\CL,\qquad \CL^\circ=\CR.
$$
Theorem I.5.8 of \cite{DixvN} says, that on a Hilbert algebra $\CH_0$ there exists a faithful, semi-finite normal trace $\tau_r$ on $\CR$ and likewise $\tau_l$ on $\CL$ such that
$$
\tau_r(R_y^*R_x)=(x\mid y)\quad\text{and}\quad \tau_l(L_y^*L_x)=(x\mid y).
$$

\begin{lemma}\label{lem.5.6}
Let $G$ be a unimodular locally compact group and let $\pi\in\what G$.
Assume that $\pi\otimes\pi'$ can be realized in $\CE$.
Then $\pi$ is trace class.
\end{lemma}

After Lemmas \ref{lem3.2} - \ref{lem3.3}, this is the last assertion of Theorem \ref{thm3.1}.
So with its proof, the proof of Theorem \ref{thm3.1} will be complete.

\begin{proof}[Proof of Lemma \ref{lem.5.6}]
Let $(\pi,V_\pi)$ be in $\what G$ and let $j:V_\pi\hat\otimes V_{\pi'}\to \CE$ a continuous linear $G\times G$-map.
Let $D$ denote the  reproducing distribution of $j$.
A calculation shows that $jj^*(f)=D*f$ for $f\in C_c^\infty(G)$.
As $j$ is injective, this implies that the kernel $J$ of $j^*$ is a right ideal in the convolution algebra $C_c^\infty(G)$.
On the other hand, from 
\begin{align*}
\sp{g|jj^*(f)}&=\sp{j^*(f),j^*(\ol g)}\\
&=D(f*g^\vee)=D(g^\vee *f)=\ol{D(f^**\ol g)}\\
&=\ol{\sp{g^*|jj^*(f^*)}}
\end{align*}
we deduce that the ideal $J$ is stable under the involution *, hence a two-sided *-ideal, 
so $\CH_0=j^*(C_c^\infty(G))\subset V_\pi\hat\otimes V_{\pi'}$ inherits the structure of a *-algebra.
The inner product of $V_\pi\otimes V_{\pi'}$ turns $\CH_0$ into a Hilbert algebra as is easily verified.
As an example, we verify that $(xy\mid z)=(x\mid zy^*)$ holds. Let $x=j^*(f)$, $y=j^*(g)$ and $z=j^*(h)$, then the claim follows from
$$
\sp{j^*(f*g),j^*(h)}=D(f*g*h^*)=D(f*(h*g^*)^*)=\sp{j^*(f),j^*(h*g^*)}.
$$
As $j^*$ is continuous, we can interchange it with the convolution integral to get
$$
j^*(f)\cdot j^*(g)=j^*(f*g)=\int_G f(y)j^*(L_yg)\,dy=[\pi(f)\otimes\Id]j^*(g).
$$
Therefore the von Neumann algebra $\CL$ is generated by the operators $\pi(f)\otimes \Id$ and also by the operators $\pi(y)\otimes\Id$ for $y\in G$, by which it follows that $\CL$ is a factor, and indeed, it is isomorphic with $\CB(V_\pi)$, therefore the essentially unique faithful semi-finite normal trace $\tau_l$ coincides with a positive multiple of the standard trace on $\CB(V_\pi)$.
So there exists $c>0$ such that for $f\in C_c^\infty(G)$ one has
$$
\tr\pi(f)^*\pi(f)=c\tau_l(L_f^*L_f)=c(f\mid f)<\infty
$$
therefore $\pi(f)$ is a Hilbert-Schmidt operator and the claim follows.
The lemma and therefore the theorem are proven.
\end{proof}

\section{Multiplicity free pairs}
Let $G$ be a locally compact group and $K$ a compact  subgroup.
Any irreducible unitary representation $\pi$ of $G$ decomposes, when restricted to $K$, as a direct sum of irreducible representations $\tau$ of $K$.
Thus we have a well-defined multiplicity, denoted $m(\pi,\tau)$.
The pair $(G,K)$ is called a \e{multiplicity free pair}, if  $m(\pi,\tau)\le 1$ holds for all $\pi\in\what G$ and all $\tau\in\what K$.
It is called a \e{Gelfand pair}, if $m(\pi,\triv)\le 1$ for each $\pi\in\what G$.
The following proposition is in \cite{Koorn}.

\begin{proposition}
[Koornwinder]
Consider $K$ as a subgroup of $G\times K$ via the diagonal embedding.
Then the pair $(G,K)$ is multiplicity free if and only if the pair $(G\times K,K)$ is a Gelfand pair.
\end{proposition}

\begin{proof}
As the representations are finite-dimensional, we have for $\ga,\tau\in\what K$ that $\ga\otimes\tau\cong \Hom_\C(\ga,\tau^*)$ and so  
$$
\dim(\ga\otimes\tau)^K=\begin{cases}1&\ga=\tau^*,\\
0&\ga\ne\tau^*.\end{cases}
$$
Any irreducible representation of $G\times K$ is a tensor product $\pi\otimes\tau$ of some $\pi\in\what G$ and $\tau\in\what K$.
The proposition now follows from 
\begin{align*}
\dim (\pi\otimes\tau)^K=\sum_{\ga\in\what K} m(\pi,\ga)\dim(\ga\otimes\tau)^K=m(\pi,\tau^*).\mqed
\end{align*}
\end{proof}

We now consider the case of a Lie group $G$ with finitely many connected components and an arbitrary closed subgroup $H$.
For $(\pi,V_\pi)\in\what G$ let $V_\pi^\infty$ be the space of smooth vectors, i.e., the space of all $v\in V_\pi$ such that the map $G\to V_\pi$, $x\mapsto\pi(x)v$ is infinitely differentiable.
This space is a Fr\'echet space equipped with the semi-norms
$$
N_D(v)=\norm{\pi(D)v},
$$
where $D$ runs in $U(\g)$, the universal enveloping algebra of the Lie algebra $\g$ of $G$.

By the Theorem of Dixmier and Malliavin \cite{DM}, we have that the space $V_\pi^\infty$ equals $\pi(C_c^\infty(G))V_\pi$.
We denote by $\pi^\infty$  the restriction of $\pi$ to $V_\pi^\infty$.
For $(\tau,V_\tau)\in\what H$ let
$$
m(\pi,\tau)=\dim\Hom_{H,\cont}(V_\pi^\infty,V_\tau),
$$
where $\Hom_{H,\cont}(V_\pi^\infty,V_\tau)$ denotes the space of continuous linear maps commuting with the $H$-action.

The pair $(G,H)$ is called \e{multiplicity free}, if $m(\pi,\tau)\le 1$ for all $\pi\in\what G$ and all $\tau\in\what H$.
It is called a \e{Gelfand pair} if $m(\pi,\triv)\le 1$ for all $\pi\in\what G$.

\begin{theorem}
Let $G$ be a unimodular Lie group with finitely many connected components and let $H$ be a closed unimodular subgroup. Assume that $G$ is trace class.
Then
\begin{center}
$(G,H)$ is multiplicity free\quad$\Leftrightarrow$\quad
$(G\times H,H)$ is a Gelfand pair.
\end{center}
\end{theorem}

\begin{proof}
This is the main result of \cite{Multfree}.
\end{proof}

{\bf Remark.}
The last theorem also allows for examples like the Mautner group inside $\C^2\rtimes\R^2$.
More precisely, let $G=\C^2\rtimes\R^2$, where $\R^2$ acts on $\C^2$ via the matrix representation $(s,t)\mapsto \smat{e^{2\pi is}}\ \ {e^{2\pi i t}}$. Choose $\al,\beta\in\R^\times$ with $\al/\beta\notin\Q$ and let $H=\C^2\rtimes L\subset G$, where $L=\R(\al,\beta)\subset\R^2$.
Then $G$ is Mackey-regular but $H$ is not and $H$ is not type I.
But the pair $(G,H)$ is multiplicity-free, as one concludes from the fact that for $(z,w)\in\C^2$ the $L$-orbit $L(z,w)$ is dense in the $\R^2$-orbit.

{\small {\it Anton Deitmar:} Math. Inst.,
Auf der Morgenstelle 10,
72076 T\"ubingen,
Germany.
{\tt deitmar@uni-tuebingen.de}\\
{\it Gerrit van Dijk:} 
Math. Inst., Niels Bohrweg 1, 2333 CA Leiden, The Netherlands. {\tt dijk@math.leidenuniv.nl}}

\end{document}